\documentclass[leqno]{article}

\usepackage{graphicx}

\usepackage{amsmath,amsthm,amsfonts,amssymb,latexsym,amscd,enumerate}

\usepackage{xy}
\xyoption{all}
\usepackage{boxedminipage}

\theoremstyle{plain}
    \newtheorem{theorem}{Theorem}
    \newtheorem{lemma}[theorem]{Lemma}
    \newtheorem{corollary}[theorem]{Corollary}
    \newtheorem{proposition}[theorem]{Proposition}

 \theoremstyle{definition}

    \newtheorem{remark}[theorem]{Remark}

\theoremstyle{remark}

\numberwithin{equation}{section}

    \newcommand{\R}{\mathbb{R}}
    \newcommand{\C}{\mathbb{C}} 
    \newcommand{\N}{\mathbb{N}}
    \newcommand{\Z}{\mathbb{Z}}

\newcommand{\kg}{\mathfrak{g}} 
\newcommand{\kk}{\mathfrak{k}}  
\newcommand{\kp}{\mathfrak{p}} 

\newcommand{\cS}{\mathcal{S}}

\DeclareMathOperator{\SO}{SO}

\DeclareMathOperator{\Spin}{Spin}

\newcommand{\Spinc}{\Spin^c}
 \newcommand{\dirac}{\partial \!\!\! \slash}

\DeclareMathOperator{\Ind}{Ind}

\DeclareMathOperator{\Ad}{Ad}

\DeclareMathOperator{\DInd}{D-Ind}
\DeclareMathOperator{\KInd}{K-Ind}

\DeclareMathOperator{\ind}{index}

\newcommand{\red}{r}
\newcommand{\Aav}{\hat A_G}

\begin{document}

\title{$\Spin$-structures and proper group actions}

\author{Peter Hochs\footnote{
University of Adelaide, \texttt{peter.hochs@adelaide.edu.au}} 
\hspace{1mm} 
and Varghese Mathai\footnote{
University of Adelaide, \texttt{mathai.varghese@adelaide.edu.au}}}  
\date{}


\maketitle

\begin{abstract}
We generalise Atiyah and Hirzebruch's vanishing theorem for actions by compact groups on compact $\Spin$-manifolds to possibly noncompact groups acting properly and cocompactly on possibly noncompact $\Spin$-manifolds. As corollaries, we obtain some vanishing results for an $\hat A$-type genus.
\end{abstract}

\tableofcontents

\section{Introduction}

In 1970, Atiyah and Hirzebruch \cite{AH} proved the following remarkable result. 
\begin{theorem}\label{thm:A-H}
Let $N$ be a compact, connected even-dimensional manifold and $K$ be a compact
connected Lie group acting smoothly and non-trivially on $N$. Suppose also that $N$ has a $K$-equivariant Spin-structure. Then the equivariant index of the Dirac operator on $N$ vanishes in the representation ring of $K$,
\begin{equation}
\ind_K(\dirac_N) = 0 \quad \in R(K).
\end{equation}
In particular, the $\hat A$-genus of $N$ is zero.
\end{theorem}

Their result then inspired many, especially Witten \cite{Witten} who studied two-dimensional quantum field theories 
and the index of the Dirac operator on free loop space $LN$, relating it to the rigidity of 
certain Dirac-type operators on $N$ and the elliptic genus, which was proved in \cite{Bott-Taubes, Taubes}.

Our goal in this note is to extend Theorem  \ref{thm:A-H} to the non-compact setting. The result is Theorem \ref{thm:H-M}, which can be stated in an equivalent way as Theorem \ref{thm:H-M2}. One consequence is a result related to rigidity, Corollary \ref{cor:rigidity}. These results involve $K$-theory and $K$-homology of $C^*$-algebras, but they have purely differential geometric consequences, as noted in Corollary \ref{cor:Ahat}.

\subsection*{Acknowledgements}

The first author was supported by a Marie Curie fellowship from the European Union. The second author thanks the Australian Research Council for support via the ARC Discovery Project grant DP130103924.

\section{Results and applications}

Let $M$ be a 
manifold, on which a connected Lie group $G$ acts properly and isometrically. Suppose that the action is cocompact, i.e.\ $M/G$ is compact, and that $M$ has a $G$-equivariant $\Spin$-structure. Let 
\[
\ind_G(\dirac_M)  \in K_{\bullet}(C^*_rG)
\]
 be the equivariant index of the associated $\Spin$-Dirac operator. Here $K_{\bullet}(C^*_rG)$ is the $K$-theory of the reduced group $C^*$-algebra of $G$, and $\ind_G$ denotes the analytic assembly map used in the Baum--Connes conjecture \cite{BCH, Kasparov}. If $G$ is compact, then $K_{0}(C^*_rG) = R(G)$ and $K_1(C^*_rG)=\{0\}$, and the analytic assembly map is the usual equivariant index. 
 
 Let $K<G$ be a maximal compact subgroup, and suppose $G/K$ has a $G$-equivariant $\Spin$-structure. This is true for a double cover of $G$, as pointed out in Section \ref{sec slice}. 

The action by $G$ on $M$ will be called \emph{properly trivial} if all stabilisers are maximal compact subgroups of $G$. For a proper action, the stabilisers cannot be larger. The action is called \emph{properly nontrivial} if it is not properly trivial.

For any manifold $X$, we write $\hat A_X$ for the $\hat A$-class of $X$. If $X$ is compact, then we denote its $\hat A$-genus by
 \[
 \hat A(X) := \int_X \hat A_X.
 \]

 Atiyah and Hirzebruch's Theorem \ref{thm:A-H} generalises as follows.
\begin{theorem}\label{thm:H-M}
As above, let $G$ be a connected Lie group, with maximal compact subgroup $K$, such that $G/K$ has a $G$-equivariant $\Spin$-structure. Suppose that $G$ acts properly and cocompactly on a manifold $M$, and that $M$ has a $G$-equivariant $\Spin$-structure. If the action is properly nontrivial, then
\[
\ind_G(\dirac_M) = 0.
\]
\end{theorem}
This result can be restated in an equivalent form as follows.
\begin{theorem}\label{thm:H-M2}
Consider the setting of Theorem \ref{thm:H-M}.
One has $\ind_G(\dirac_M) \not= 0$ if and only if there is a compact $\Spin$-manifold $N$ with $\hat A(N) \not=0$, and a $G$-equivariant diffeomorphism
\[
M \cong G/K \times N,
\]
where $G$ acts trivially on $N$.
\end{theorem}

A first consequence of Theorem \ref{thm:H-M} is a rigidity-type result. This involves the \emph{Dirac induction} map 
\begin{equation} \label{eq Dirac ind}
\DInd_K^G\colon R(K) \to K_{\bullet}(C^*_rG). 
\end{equation}
Here $R(K)$ is the representation ring of $K$. This map is an isomorphism of Abelian groups by the Connes--Kasparov conjecture, which was proved for (almost) connected groups in \cite{CEN}. 
\begin{corollary} \label{cor:rigidity}
In the setting of Theorem \ref{thm:H-M}, one has
\[
\ind_G(\dirac_M) \in \Z \cdot \DInd_K^G[\C],
\]
where $[\C] \in R(K)$ is the class of the trivial representation.
\end{corollary}

Theorems \ref{thm:H-M} and \ref{thm:H-M2} and Corollary \ref{cor:rigidity}  have consequences within differential geometry, not involving $K$-theory and $K$-homology. Let $c\in C^{\infty}_c(M)$ be a cutoff function, that is to say a non-negative function satisfying 
\[
\int_Gc(g^{-1}m)\mathrm{d}g=1
\]
 for all $m\in M$, for a fixed left Haar measure $\mathrm{d}g$ on $G$. The \emph{averaged $\hat A$-genus} of the action by $G$ on $M$ is
 \[
 \Aav(M) := \int_M c\hat A_M.
 \]
It was shown to be independent of $c$ in \cite{Wang}. 

Since $\hat A_M$ is an even cohomology class, the averaged $\hat A$-genus of an odd-dimensional manifold is zero. For even-dimensional $\Spin$-manifolds, we obtain the following results.
\begin{corollary} \label{cor:Ahat}
Consider the setting of Theorem \ref{thm:H-M}, and suppose that $M$ is even-dimensional.
\begin{enumerate}
\item If the action is properly nontrivial, then $\Aav(M) = 0$.
\item If $\Aav(M) \not= 0$, then there is a compact $\Spin$-manifold $N$ with $\hat A(N) \not=0$, and a $G$-equivariant diffeomorphism
\[
M \cong G/K \times N,
\]
where $G$ acts trivially on $N$.
\item The averaged $\hat A$-genus $\Aav(M)$ is an integer multiple of $\Aav(G/K)$.
\end{enumerate}
\end{corollary}
The first two points in this corollary provide a criterion, namely nonvanishing of the number $\hat A_G(M)$, for the action to be of a particularly simple type. In the third point, the number $\Aav(G/K)$ only depends on the group. Therefore, that point is a divisibility property of the topological invariant $\Aav(M)$ of the action.

\begin{remark}
There are many group actions that satisfy the hypotheses of Theorem \ref{thm:H-M}. Indeed, 
if $K$ acts on a compact $\Spin$-manifold $N$ as in Theorem \ref{thm:A-H}, then Theorem \ref{thm:H-M} applies to the action  by $G$ on the fibred product $G\times_K N$, as we will see.
If $K = S^1$, then it is proved in the theorem in Section 2.3 in \cite{AH} 
that any compact oriented manifold $X$ with $\displaystyle \hat A(X)=0$ has the property that the disjoint union of $r$ copies of $X$, 
$r X$ (for some $r \in \N$) is oriented cobordant to a compact Spin manifold $N$ which has a non-trivial $S^1$-action on each of its components. Then the action by $K$ on $N$ satisfies the hypotheses of Theorem \ref{thm:A-H}, so that the action  of $G$ on $G\times_K N$ satisfies the conditions of  Theorem \ref{thm:H-M}. 
\end{remark}

\begin{remark}
In the setting of Theorems \ref{thm:H-M} and \ref{thm:H-M2}, there is a proper, equivarant map $p\colon M \to G/K$ (see Theorem \ref{thm:abels} below). The map $p_*$ induced on $K$-homology relates the equivariant indices on $M$ and $G/K$ by the diagram
\[
\xymatrix{
K^G_\bullet(M) \ar[r]^-{\ind_G} \ar[d]_-{p_*} & K_{\bullet}(C^*_rG). \\
K^G_\bullet(G/K) \ar[ur]_-{\ind_G}^{\cong} &
}
\]
Since the Baum--Connes conjecture is true for connected groups by
 Theorem 1.1 in  \cite{CEN}, the equivariant index on $G/K$ defines an isomorphism $$K^G_\bullet(G/K) \cong  K_{\bullet}(C^*_rG).$$ Hence $p_*[\dirac_M] = 0$ if and only if $\ind_G(\dirac_M) = 0$, so the non-vanishing of $\ind_G(\dirac_M)$ in Theorems \ref{thm:H-M} and \ref{thm:H-M2}, can be replaced by the non-vanishing of the class $p_*[\dirac_M]$. This removes $K$-theory and the assembly map from these results.
\end{remark}

\section{$\Spin$-structures on slices} \label{sec slice}

We begin by recalling the smooth version of Abels' slice theorem for proper group actions. Let $M$ be a smooth manifold, and let $G$ be a connected Lie group acting properly on $M$. Let $K<G$ be maximal compact.
\begin{theorem}[p.\ 2 of \cite{Abels}]\label{thm:abels}
There is a smooth, $K$-invariant submanifold $N \subset M$, such that 
the map $[g, n] \mapsto gn$ is a $G$-equivariant diffeomorphism
\begin{equation} \label{eq Abels fibr}
G\times_K N \cong M.  
\end{equation}
Here the left hand side is the quotient of $G\times N$ by the action by $K$ given by
\[
k\cdot (g, n) = (gk^{-1}, kn),
\]
for $k \in K$, $g \in G$ and $n \in N$. 
\end{theorem}
We call \eqref{eq Abels fibr} an associated Abels fibration of $M$, as it is a fibre bundle over $G/K$ with fibre $N$. From now on, fix $N$ as in Theorem \ref{thm:abels}.

The fixed point set $N^K$ of the action by $K$ on $N$ is related to the action by $G$ on $M$ in the following way.
\begin{lemma} \label{lem fixed pts}
One has
\[
M_{(K)} = G \cdot N^K \cong G/K \times N^K,
\]
where $M_{(K)}$ is the set of points in $M$ with stabilisers conjugate to $K$.
\end{lemma}

\begin{proof}
Let $m \in M_{(K)}$, and write $m = [g, n]$ for $g \in G$ and $n \in N$, under the correspondence \eqref{eq Abels fibr}. Then 
the stabiliser subgroups $G_m$ of $G$ and $K_n$ of $K$ satisfy $G_m = gK_n g^{-1}$. So $G_m$ is conjugate to $K$ if and only if $K_n$ is. Since $K_n < K$, it is conjugate to $K$ precisely if it equals $K$. 
\end{proof}

Now fix a $K$-invariant inner product on the Lie algebra $\kg$ of $G$, and let $\kp \subset \kg$ be the orthogonal complement to the Lie algebra $\kk$ of $K$. Suppose $\Ad: K \to \SO(\kp)$ lifts to
\begin{equation} \label{eq tilde Ad}
 \widetilde{\Ad}: K \to \Spin(\kp).
 \end{equation}
  This is always possible if one replaces $G$ by a double cover. Indeed, consider the diagram
\[
\xymatrix{
\widetilde{K} \ar[r]^-{\widetilde{\Ad}} \ar[d]_{\pi_K}& \Spin(\kp) \ar[d]_{\pi}^{2:1}\\
K \ar[r]^-{\Ad} & \SO(\kp),
}
\]
where
\[
\widetilde{K} := \{ (k, a) \in K\times \Spin(\kp); \Ad(k) = \pi(a)\},
\]
and the maps $\pi_K$ and $\widetilde{\Ad}$ are defined by
\[
\begin{split}
\pi_K(k, a)&:= k; \\
\widetilde{\Ad}(k, a) &:= a,
\end{split}
\]
for $k \in K$ and $a \in \Spin(\kp)$. Then for all $k \in K$,
\[
\pi_K^{-1}(k) \cong \pi^{-1}(\Ad(k)) \cong \Z_2,
\]
so $\pi_K$ is a double covering map. Since $G/K$ is contractible, $\widetilde{K}$ is the maximal compact subgroup of a double cover of $G$.

Suppose $M$ has a $G$-equivariant $\Spin$-structure $P_M \to M$.
In Section 3.2 of \cite{HochsDS}, Section 3.2 of \cite{HM} and also \cite{HMAIM}, an induction procedure of equivariant $\Spinc$-structures from $N$ to $M$ is described, which we will denote by $\Ind_N^M$ here. We will use the fact that any $G$-equivariant $\Spin$-structure on $M$ can be obtained via this induction procedure. (See also Proposition 3.10 in \cite{HM}.)
\begin{lemma} \label{lem Spin slice}
Suppose that $G/K$ and $M$ have $G$-equivariant $\Spin$-structures. 
Then
there is a $K$-equivariant $\Spin$-structure $P_N \to N$ such that
$
 \Ind_N^M(P_N)
$
is the original $\Spin$-structure on $M$.
\end{lemma}
\begin{proof}
Let $\kp_N \to N$ be  the trivial vector bundle $N \times \kp \to N$, with the diagonal $K$-action. 
Since $G/K$ has an equivariant $\Spin$-structure, the bundle $\kp_N$ has the 
$K$-equivariant $\Spin$-structure
\[
N\times \Spin(\kp) \to N
\]
 where $K$ acts diagonally on $N\times \Spin(\kp)$ via the lift \eqref{eq tilde Ad} of the adjoint action by $K$ on $\kp$.
Since
\[
TM  = G\times_K(TN \oplus \kp_N)
\]
(see Proposition 2.1 and Lemma 2.2 in \cite{HochsDS}), the restriction
$P_M|_N$ is a $K$-equivariant $\Spin$-structure on 
\[
TM|_N = TN \oplus \kp_N.
\] 

Togther with the $\Spin$-structure on $\kp_N$, this determines a $K$-equivariant $\Spin$-structure on $TN$ by the two-out-of-three lemma. Explicitly, if $\cS_M$ is the spinor bundle associated to $P_M$, and $\Delta_{\kp}$ is the standard representation of $\Spin(\kp)$, then the spinor bundle $\cS_N$ associated to $P_N$ is determined by
\[
\cS_M|_N = \cS_N \otimes \Delta_{\kp}.
\]
The Clifford action $c_{\cS_N}$ by $TN$ on $\cS_N$ is determined by the property that for all $n \in N$, $v \in T_nN$ and $X \in \kp$,
\[
c_{\cS_M|_N}(v, X) = \left\{ \begin{array}{ll} 
c_{\cS_N}(v) \otimes 1_{\Delta_{\kp}} + \varepsilon_{\cS_N} \otimes c_{\kp}(X) & \text{if $\dim G/K$ is even;} \\
c_{\cS_N}(v) \otimes 1_{\Delta_{\kp}} + \varepsilon_{\cS_N} \otimes \varepsilon_{\Delta_{\kp}} c_{\kp}(X) & \text{if $\dim G/K$ is odd.}
\end{array}\right.
\]
Here $\varepsilon_{\cS_N}$ and $\varepsilon_{\Delta_{\kp}}$ are the grading operators on $\cS_N$ and $\Delta_{\kp}$, respectively. See also Section 3.1 in \cite{Plymen}.

In Lemma 3.9 of \cite{HM}, it is shown that 
\[
P_M = \Ind_N^M(P_N).
\]
\end{proof}

\begin{remark}
We have only considered the principal bundle part $P_M \to M$ of a $\Spin$-structure on $M$, not the isomorphism
\[
P_M \times_{\Spin(\dim M)} \R^{\dim M} \cong TM.
\] 
This isomorphism determines the Riemannian metric on $M$ induced by the $\Spin$-structure. The index of $\dirac_M$ is independent of this metric, however, so Lemma \ref{lem Spin slice} is enough for our purposes.
\end{remark}

\section{Proofs of the results}

The \emph{quantisation commutes with induction} techniques of \cite{HochsDS, HochsPS},  adapted to the $\Spin$-setting, allow us to deduce Theorems \ref{thm:H-M} and \ref{thm:H-M2}
from Atiyah and Hirzebruch's Theorem \ref{thm:A-H}. 
This is based on the fact that the Dirac induction map \eqref{eq Dirac ind} relates the equivariant indices of the $\Spin$-Dirac operators $\dirac_N$ on $N$ and $\dirac_M$ on $M$, associated to the $\Spin$-structures $P_N$ and $P_M$, respectively, to each other. (See also Theorem 5.7 in \cite{HM}.)
\begin{proposition} \label{prop quant ind}
If $G/K$ has a $G$-equivariant $\Spin$-structure, then
\[
\DInd_K^G\bigl( \ind_K(\dirac_N)\bigr) = \ind_G(\dirac_M) \quad \in K_{\bullet}(C^*_{r}G).
\]
\end{proposition}
\begin{proof}
Let $K_{\bullet}^K(N)$ and $K_{\bullet}^G(M)$ be the equivariant $K$-homology groups \cite{BCH} of $N$ and $M$, respectively.
In Theorem 4.6 in \cite{HochsDS} and Theorem 4.5 in \cite{HochsPS}, a map
\[
\KInd_K^G\colon K_{\bullet}^K(N) \to K_{\bullet}^G(M)
\]
is constructed, such that the following diagram commutes:
\[
\xymatrix{
K_{\bullet}^G(M) \ar[r]^-{ \ind_G} & K_{\bullet}(C^*_{\red}G) \\
K_{\bullet}^K(N) \ar[u]^-{\KInd_K^G} \ar[r]^-{\ind_K} & R(K) \ar[u]_-{\DInd_K^G}.
}
\]
In Section 6 of \cite{HochsDS}, it is shown that the $K$-homology class of a $\Spinc$-Dirac operator on $N$, associated to a connection $\nabla^N$ on the determinant line bundle of a $\Spinc$-structure, is mapped to the class  of a $\Spinc$-Dirac operator on $M$ associated to a connection $\nabla^M$ induced by $\nabla^N$ on the determinant line bundle of the induced $\Spinc$-structure, by the map $\KInd_K^G$. In the $\Spin$-setting, both connections $\nabla^N$ and $\nabla^M$ are trivial connections on trivial line bundles. Hence one gets
\[
\KInd_K^G[\dirac_N] = [\dirac_M],
\] 
and the result follows.
\end{proof}

\medskip
\noindent \emph{Proof of Theorems \ref{thm:H-M} and \ref{thm:H-M2}.}
Consider the setting of Theorem \ref{thm:H-M}, let $N\subset M$ be as in Theorem \ref{thm:abels}. Consider a $K$-equivariant $\Spin$-structure on $N$ as in Lemma \ref{lem Spin slice}. By Proposition \ref{prop quant ind}, we have
\begin{equation} \label{eq quant ind pf}
\ind_G(\dirac_M) = \DInd_K^G\bigl( \ind_K(\dirac_N)\bigr).
\end{equation}
The stabiliser of a point $m \in M$ is a maximal compact subgroup of $G$ if and only if $m \in M_{(K)}$. Hence, by Lemma \ref{lem fixed pts}, the condition on the stabilisers of the action by $G$ on $M$ is equivalent to the action by $K$ on $N$ being nontrivial. So if $N$ is even-dimensional, Theorem \ref{thm:A-H} implies that
\[
\ind_K(\dirac_N) = 0,
\]
and Theorem \ref{thm:H-M} follows. (If $N$ is odd-dimensional, this index always equals zero.)

To prove Theorem \ref{thm:H-M2}, we note that by
 injectivity of Dirac induction, the equality \eqref{eq quant ind pf} implies that 
\[
\ind_G(\dirac_M) \not=0 \quad \Leftrightarrow \quad \ind_K(\dirac_N) \not=0. 
\]
Furthermore,
\[
 \ind_K(\dirac_N) \not=0 \quad \Leftrightarrow \quad \text{$K$ acts trivially on $N$ and $\hat A(N)\not=0$}.
\]
This equivalence follows from Theorem  \ref{thm:A-H}, because if $K$ acts trivially on $N$, then 
\mbox{$\ind_K(\dirac_N)$} equals $\ind(\dirac_N) = \hat A(N)$ copies of the trivial representation. Since $K$ acts trivially on $N$ if and only if $M = G/K \times N$, the claim follows.
\hfill $\square$

\begin{remark}
In the proof of Theorem \ref{thm:A-H} in \cite{AH}, the fixed point set of the action plays an important role. Because proper actions by noncompact groups do not have fixed points, the authors think it is unlikely that the arguments in \cite{AH} generalise directly to a proof of Theorems \ref{thm:H-M} and \ref{thm:H-M2}. This is one reason to apply the induction argument used here.
\end{remark}

It remains to prove Corollaries \ref{cor:rigidity} and \ref{cor:Ahat}.

\medskip \noindent
\emph{Proof of Corollary \ref{cor:rigidity}.}
If $\ind_G(\dirac_M) = 0$, the result holds trivially.  If \hbox{$\ind_G(\dirac_M)$} is nonzero, then we have the decomposition  $M \cong G/K \times N$ of Theorem \ref{thm:H-M2}.
By multiplicativity of the analytic assembly map (see Theorem 5.2 in \cite{HochsDS}), we then conclude that
\begin{equation} \label{eq ind M N GK}
\ind_G(\dirac_M) = \ind(\dirac_N) \ind_G(\dirac_{G/K}) = \ind(\dirac_N) \DInd_K^G[\C].
\end{equation}
\hfill $\square$

\medskip\noindent
\emph{Proof of Corollary \ref{cor:Ahat}.}
 Let $\tau\colon C^*_{\red}G \to \C$ be the von Neumann trace, determined by
 \[
 \tau\bigl(R(f)^*R(f) \bigr)=\int_G |f(g)|^2 \mathrm{d}g,
 \]
 for $f\in L^1(G) \cap L^2(G)$, where $R$ denotes the right regular representation. This induces a morphism $\tau_* \colon K_0(C^*_{\red}G) \to \R$. Since $M$ is even-dimensional, we have $\ind_G(\dirac_M) \in K_0(C^*_rG)$.  
 Wang showed in Theorem 6.12 in \cite{Wang} that
 \[
 \Aav(M) = \tau_*(\ind_G(\dirac_M)).
 \]
 So $\ind_G(\dirac_M) = 0$ implies $\Aav(M)=0$. Furthermore,
 \[
 \tau_* \bigl( \DInd_K^G[\C]\bigr) = \Aav(G/K).
 \] 
 Hence Theorem \ref{thm:H-M}, Theorem  \ref{thm:H-M2} and Corollary \ref{cor:rigidity} imply Corollary \ref{cor:Ahat}.
\hfill $\square$


\end{document}